\documentclass[makeidx,a4paper,11pt]{article}
\usepackage{makeidx}
\usepackage{amssymb , amsmath, amsthm}
\newtheorem{defi}{Definition} 
\newtheorem{thm}[defi]{Theorem}

 \newtheorem{prop}[defi]{Proposition}
\newtheorem{lemme}[defi]{Lemma}
\newtheorem{cor}[defi]{Corollary}

\newcommand{\matrice}{\begin{pmatrix}}
\newcommand{\ok}{\end{pmatrix}}
\newcommand{\dmatrice}{\begin{vmatrix}}
\newcommand{\dok}{\end{vmatrix}}
 
\newcommand{\twosystem}[2]{\left\{\begin{aligned} &#1\\ &#2\end{aligned}\right.}
\newcommand{\threesystem}[3]{\left\{ \begin{aligned}&#1\\ &#2\\&#3\end{aligned}\right.}

\newcommand{\nero}{\smallskip$\bullet\quad$\rm}

\newcommand{\parte}[1]{\smallskip\noindent {\rm#1)}\,\,}

\newcommand{\scal}[2]{\langle{#1},{#2}\rangle}

\newcommand{\abs}[1]{\lvert{#1}\rvert}

\newcommand{\reals}{{\bf R}}
\newcommand{\sphere}[1]{{\bf S}^{#1}}

\newcommand{\real}[1]{{\bf R}^{#1}}

\newcommand{\bd}{\partial}

\newcommand{\derive}[2]{\dfrac{\bd #1}{\bd#2}}

\newcommand{\deriven}[3]{\dfrac{\bd^{#1} #2}{\bd #3^{#1}}}

\newcommand{\level}[2]{#1^{-1}(#2)}

\begin{document}

\title{Geometric rigidity of constant heat flow
\footnote{Classification AMS $2000$: 58J50, 35P15 \newline 
Keywords: Heat equation, constant flow property, overdetermined problems, isoparametric tubes}}

\author{Alessandro Savo}
\date{}

\maketitle
\begin{abstract} 
Let $\Omega$ be a compact Riemannian manifold with smooth boundary and let $u_t$ be the solution of the heat equation on $\Omega$, having constant unit initial data $u_0=1$
and Dirichlet boundary conditions ($u_t=0$ on the boundary, at all times). If at every time $t$ the normal derivative of $u_t$ is a constant function on the boundary, we say that $\Omega$ has the {\it constant flow property}. This gives rise to an overdetermined parabolic problem, and our aim is to classify the manifolds having this property. In fact, if the metric is analytic, we prove that 
$\Omega$ has the constant flow property if and only if it  is an {\it isoparametric tube}, that is, it is a solid  tube of constant radius around a closed, smooth, minimal submanifold, with the additional property that all equidistants to the boundary (parallel hypersurfaces) are smooth and have constant mean curvature. 
Hence,  the constant flow property can be viewed as an analytic counterpart to the isoparametric property.  
Finally, we relate the constant flow property with other overdetermined problems, in particular, the well-known Serrin problem on the mean-exit time function, and discuss a counterexample involving minimal free boundary immersions into Euclidean balls. 
\end{abstract}
\large

\section{Main results}

In Riemannian geometry, an overdetermined problem gives rise to the following question: is it possible to identify the geometry of a domain $\Omega$ in a Riemannian manifold assuming the existence of a solution $u$ of a certain PDE such that both $u$ and its normal derivative are constant on the boundary of $\Omega$ ?  
Perhaps the most famous example of overdetermined problem is the so-called {\it Serrin problem} :

\begin{equation}\label{serrin}
\twosystem
{\Delta v=1\quad\text{on}\quad \Omega,}
{v=0,\, \derive v{\nu}={\rm const.}\quad\text{on}\quad \bd\Omega.}
\end{equation}

J. Serrin celebrated rigidity result \cite{Ser} states that the only compact Euclidean domains supporting a solution to \eqref{serrin} are Euclidean balls.  Another famous problem  is the so-called {\it Schiffer problem} 
\begin{equation}\label{schiffer}
\twosystem
{\Delta u=\lambda u\quad\text{on}\quad\Omega}
{u={\rm const}, \, \derive u{\nu}=0\quad\text{on}\quad\bd\Omega;}
\end{equation}
the Schiffer conjecture states that the only compact Euclidean domains supporting a non-trivial solution to \eqref{schiffer} for some eigenvalue $\lambda>0$ are balls. It is well-known that proving this conjecture is equivalent to solving the famous Pompeiu problem (see \cite{BS}, \cite{Wil}).  
Only partial solutions are known (among them, see \cite{BY}) and we refer the reader to the papers \cite{Ber}, \cite{Ber2}, \cite{BS} for related results.
We remark that not much is known about these problems for domains in a general Riemannian manifold.

The study of overdetermined problems is a very active and interesting field of research, lying at the border between geometry and analysis;   for an overview, see for example \cite{Shk}, and then  \cite{Cha}, \cite{ESI}, \cite{FMW}, \cite{GS}, \cite{KN}, \cite{Liu}, \cite{Mol}, \cite{S3}, \cite{Sch}, \cite{Sou}, \cite{Wei}, \cite{Wil}, \cite{Wil2}; for problems in Riemannian manifolds see for example \cite{FM} and \cite{FMV}. We stress that we assume compactness of $\Omega$ in this paper.  The non-compact situation (for example, exterior domains in Euclidean space) is quite rich and interesting, and we refer for example to \cite{dPPW}, \cite{RosSic} and the  preprint  \cite{RRS}. 
The list is very incomplete, due to the many interesting contributions to this  problem since Serrin's seminal paper \cite{Ser}. 

\smallskip

In this paper we classify  compact Riemannian manifolds with analytic metric and smooth boundary satisfying a certain overdetermined problem for the heat kernel (defined in  \eqref{overdet}) : we show that the class of such manifolds (which are said to have the {\it constant flow property}) coincides with the class of the so-called {\it isoparametric tubes} (see Definition \ref{iso}). This generalizes to Riemannian manifolds the results of \cite{S3}, obtained in the standard sphere. Thus, this is one case in which it is possible to give a precise description, in the general Riemannian setting, of the geometry of manifolds supporting a solution to the given overdetermined problem, so that the constant flow property \eqref{overdet} could be seen as an analytic counterpart to the isoparametric property, very much studied in differential geometry.  

\smallskip

Let us see the contents of this introduction. In Section \ref{cfp} we define and discuss the overdetermined problem at hand and the class of isoparametric tubes, while in Section \ref{previousresults} we recall the main results from \cite{S3} . In Section \ref{mainresult} we state our main equivalence result and in Section \ref{cfphar} we recall that if a manifold has the constant flow property  then it satisfies also the classical Serrin problem \eqref{serrin} ; then, we prove that the converse does not hold for the class of (minimal) free boundary immersion into a Euclidean $3$-ball having more than $2$ boundary components. 


\subsection{The constant flow property}\label{cfp}

Let $(\Omega^n,g)$ be a compact Riemannian manifold of dimension $n$ with smooth boundary $\bd\Omega$.  Consider the solution $u=u(t,x):[0,\infty)\times \Omega\to \reals$ of the heat equation on $\Omega$ with initial data $1$ and Dirichlet boundary conditions:
\begin{equation}\label{temp}
\threesystem
{\Delta u+\derive ut=0\quad\text{on}\quad \Omega,}
{u(0,x)=1\quad\text{for all}\quad x\in\Omega,}
{u(t,y)=0\quad\text{for all $y\in\bd\Omega$ \, and $t>0$},}
\end{equation}
where $\Delta$ is the Laplace-Beltrami operator defined by the Riemannian metric $g$ and acting on the space variable $x$. We will often write $u(t,x)$ as  $u_t(x)$ so that $u_0=1$. The function $u$ is a basic object in heat diffusion : in fact it can be written
\begin{equation}\label{hk}
u(t,x)=\int_{\Omega}k(t,x,y)dy,
\end{equation}
where $k: (0,\infty)\times\Omega\times\Omega\to\reals $ is the heat kernel of $\Omega$ (that is, the fundamental solution of the heat equation with Dirichlet boundary conditions). About  the physical meaning, $u(t,x)$ is the temperature at time $t$, at the point $x\in\Omega$, assuming that the initial temperature distribution is constant, equal to $1$, and that the boundary $\bd\Omega$ is subject to absolute refrigeration. 

Now let $\nu$ be the unit normal vector field of $\bd\Omega$, pointing inward, and let $y\in\bd\Omega$.  Then, $\derive{u}{\nu}(t,y)$ can be interpreted as the {\it heat flow} at time $t$, at the boundary point $y$.  A complete asymptotic expansion for the heat flow $\derive{u_t}{\nu}$ at any fixed boundary point has been obtained in \cite{S2} (this result was used in the paper \cite{S3}). 

\medskip

\begin{defi}\label{chf} We say that $\Omega$ has the {\rm constant flow property} if, for all fixed  $t>0$, the heat flow
$$
\derive{u}{\nu}(t,\cdot):\bd\Omega\to\reals
$$ 
is a constant function on $\bd\Omega$.
\end{defi}

In other words, a manifold has the constant flow property if and only if it supports a solution to the following overdetermined problem:
\begin{equation}\label{overdet}
\threesystem
{\Delta u+\derive ut=0\quad\text{on}\quad \Omega,}
{u(0,x)=1\quad\text{for all}\quad x\in\Omega,}
{u(t,y)=0, \, \derive u{\nu}(t,y)=c(t)\quad\text{for all $y\in\bd\Omega$ \, and $t>0$},}
\end{equation}
where $c(t)$ is a function depending only on $t$. 
Manifolds with the constant flow property are {\it perfect heat diffusers}, as defined in the introduction of \cite{S3} (see Theorem 9 in \cite{S3} for a characterization in terms of the heat content with zero mean boundary data). 

\smallskip

In this paper,  assuming that $\Omega$ is analytic, we will show the equivalence of this property with the following geometric property.

\begin{defi}\label{iso}  We say that the compact manifold with boundary  $\Omega$ is a {\rm smooth tube around $P$} if there exists a smooth, closed submanifold $P$ of $M$ and a number $R>0$ such that:
\item a) $\Omega$ is the set of points at distance at most $R$ from $P$,

\item b) For each $s\in (0,R]$, the equidistant
$$
\Sigma_s=\{x\in \Omega: d(x,P)=s\}
$$
is a smooth hypersurface of $\Omega$. 

\smallskip

We say that the smooth tube $\Omega$ is an {\rm isoparametric tube} if every equidistant $\Sigma_s$ as above has contant mean curvature.
\end{defi}

\nero The submanifold $P$ is called the {\it soul} of $\Omega$, and can have dimension 
$\dim P=0,\dots,n-1$. The soul is then  an embedded submanifold.

\nero For example, a solid revolution  torus in $\real 3$ with radii $a>b>0$ is a smooth tube (the soul $P$ is a circle), but is not an isoparametric tube because equidistants have variable mean curvature. In fact, the only (compact) isoparametric tubes in Euclidean spaces are the balls, in which case $P$ reduces to a point.  This follows from the general fact that the soul of an isoparametric tube is always a minimal submanifold (see Theorem \ref{min} below)  and in Euclidean space the only compact minimal submanifolds are points. 

\nero We don't assume that the boundary of an isoparametric tube $\Omega$  is connected. In fact, it is easy to show that $\bd\Omega$ can have at most two boundary components (see Proposition \ref{up});  moreover, if $\bd\Omega$ has two components, as the mean curvature is constant on the boundary, it must take the same value on each of the two components.  In particular, a domain in a standard sphere, bounded by two geodesic spheres, is an isoparametric tube if and only if the two boundary spheres are isometric and have equal (or antipodal) centers : in that case, the soul is an equatorial (i.e. totally geodesic) hypersurface. 

\nero Obvious examples of isoparametric tubes are given by geodesic balls in space forms : in that case, the soul is a point. More generally, any geodesic ball in a locally harmonic manifold is (more or less by definition) an isoparametric tube around its center.

\nero A {\it revolution manifold with boundary} is a Riemannian manifold $(\Omega,g)$
isometric to
$[a,R]\times\sphere{n-1}$ endowed with the metric $g=dr^2+\theta^{2}(r)g_{\sphere{n-1}}$,  where $g_{\sphere{n-1}}$ is the standard metric on the sphere
and $\theta^2(r)$ is a smooth, positive function on $[a,R]$. Note that $\bd\Omega$ has two components, namely  $\{a\}\times \sphere{n-1}, \{R\}\times\sphere{n-1}$. Now, rotational invariance implies that the temperature function $u_t$ defined in \eqref{temp} depends only on the radial variable $r$, which immediately implies that every such manifold has the constant flow property. 

\smallskip

If we instead assume that the metric of $\Omega$ is smooth and that there is a distinguished point $p\in\Omega$ such that $(\Omega\setminus\{p\},g)$ is isometric to
$(0,R]\times\sphere{n-1}$ endowed with the metric $g=dr^2+\theta^{2}(r)g_{\sphere{n-1}}$,  then $\bd\Omega$ has only one component, namely $\{R\}\times\sphere{n-1}$:  also this manifold has the constant flow property. 

\medskip

In the next subsection we will discuss the main class of examples of isoparametric tubes, namely, spherical domains bounded by isoparametric hypersurfaces.  We finish this section by pointing out the following fact, which is proved in \cite{GT} (and first proved  in \cite{N} when the ambient manifold is the sphere). We discuss it in more detail in the last part of the Appendix. 

\begin{thm}\label{min} The soul of an isoparametric tube is always a minimal submanifold. 
\end{thm}


\subsection{Some comments on the results of \cite{S3}}\label{previousresults} Let $\Sigma$ be a closed hypersurface of the Riemannian manifold $M$. In \cite{S3},  $\Sigma$  is called {\it isoparametric} if all parallel hypersurfaces sufficiently close to $\Sigma$ have constant mean curvature. Note that the definition is local in nature, and refers to the behavior of the mean curvature only in a neighborhood of $\Sigma$. We proved the following fact. 

\begin{thm}\label{hfone}{\rm (\cite{S3}, Theorem 2)} Let $\Omega$ be a compact domain with smooth boundary in an analytic Riemannian manifold $M$. Assume that it has the constant flow property. Then each component of $\bd\Omega$ is an isoparametric hypersurface of $M$.
\end{thm}

Isoparametric hypersurfaces were mostly studied when the ambient manifold $M$ is a space form, starting from the classical works of Segre, Cartan and M\"unzner. It is a classical fact, due to Cartan \cite{Car},  that $\Sigma$ is isoparametric if and only if it has constant principal curvatures (that is, the characteristic polynomial of the shape operator of $\Sigma$ is the same at all points). We refer to \cite{TY} and \cite{Tho} for  overviews. However, it is well-known that the only closed isoparametric hypersurfaces of Euclidean and Hyperbolic space are geodesic spheres; so, the only interesting case which remains to be discussed is that of the sphere $\sphere n$. There, we have plenty of isoparametric hypersurfaces and a beautiful result of M\"unzner shows that the number $g$ of distinct principal curvatures can only be $1,2,3,4,6$. Moreover, each $\Sigma$ is a level set of the restriction to $\sphere n$ of a suitable polynomial in $\real{n+1}$ (Cartan-M\"unzner polynomial).  See \cite{Mun}. Thus, the classification reduces to a (difficult) algebraic problem. 

Now, the constancy of the principal curvatures imply that the focal sets $M_{\pm}$ of $\Sigma$ are regular submanifolds at constant distance from $\Sigma$. The conclusion is :

\nero Any connected isoparametric hypersurface of $\sphere n$ bounds two domains $\Omega_{\pm}$, each being an isoparametric tube over the respective focal set (soul) $M_{\pm}$. Moreover, as showed by Nomizu in \cite{N}, $M_{\pm}$ are minimal submanifolds. 

\medskip

Then, in the sphere Theorem \ref{hfone} becomes :

\nero Any compact domain in $\sphere n$ having the constant flow property and connected boundary is an isoparametric tube.

\medskip

The converse statement is also true, thanks to a previous result of Shklover's.

\begin{thm}{\rm (See \cite{Shk})} \label{shklover} Let $\Omega$ be a compact spherical domain bounded by a connected isoparametric hypersurface. Then $\Omega$ has the constant flow property. 
\end{thm}

 The proof of Theorem \ref{shklover} uses a suitable ODE, coming from the existence of an isoparametric function. Combining Theorem \ref{hfone} and Theorem \ref{shklover} one finds the following characterization, which is just a restatement of Corollary 3 in \cite{S3}.
 
\begin{cor}{\rm (Corollary 3 in \cite{S3})} \label{heffeone} Let $\Omega$ be a compact domain in $\sphere n$ having connected boundary. Then, 
 $\Omega$ has the constant flow property if and only if it is an isoparametric tube. 
\end{cor}
 
We remark that Theorem \ref{hfone} is a consequence of Theorem 7 in \cite{S3}, which is valid on any smooth Riemannian manifold and will be recalled in Theorem \ref{previous} below. 

This is the state of the art. 
The scope of the present paper is to generalize the previous results and extend Corollary \ref{heffeone} from the sphere to the general Riemannian case, for analytic manifolds with smooth (not necessarily connected) boundary. 


\subsection{Main result}\label{mainresult}

Here is the main result of this paper. 

\begin{thm}\label{main} Let $\Omega$ be a compact, analytic  manifold with smooth boundary. Then, $\Omega$ has the constant flow property if and only if it is an isoparametric tube around a smooth, compact, 
connected submanifold $P$ of $\Omega$.
\end{thm}

Theorem \ref{main}  improves Theorem \ref{hfone} because it gives a description of the geometry of a domain with the constant flow property not just near its boundary but also at points far from it. This is achieved by showing that in fact the cut locus ${\rm Cut}_{\Omega}$  of the normal exponential map of the boundary of $\Omega$ is a regular submanifold, which coincides with the set of points that are at maximum distance to the boundary. Moreover, the whole domain $\Omega$ is a smooth, isoparametric tube over the soul $P\doteq {\rm Cut}_{\Omega}$.

\smallskip

In the converse statement we extend Shklover's result (Theorem \ref{shklover}) to arbitrary smooth (not necessarily analytic) isoparametric tubes ;  even in the sphere, the proof is different from 
Shklover's in the sense that it uses the procedure of averaging a function over the equidistants from the boundary, instead of using the ODE coming from an explicit isoparametric function, as in \cite{Shk}. 


\subsection{Constant flow property vs. harmonicity}\label{cfphar}

In \cite{S3} we discussed in a certain detail the relation of the constant flow property with other well-known overdetermined problems. Here we will focus on Serrin problem :
\begin{equation}\label{harmonic}
\twosystem
{\Delta v=1\quad\text{on}\quad \Omega,}
{v=0, \, \derive{v}{\nu}={\rm const}  \quad\text{on}\quad \bd\Omega.}
\end{equation}

A manifold with boundary supporting a solution to problem \eqref{harmonic} is  termed a {\it harmonic domain} in \cite{RS} because it has the following property:   for any harmonic function $h$, the mean values of $h$ on $\Omega$ and $\bd\Omega$ are the same.  In this terminology, Serrin rigidity result \cite{Ser} can be stated as follows:
 
\nero {\it Any (compact) harmonic Euclidean domain is a ball.}
 
\medskip

Then, we could ask if there is a classification of harmonic domains in the general Riemannian context, and not just in the Euclidean case.  This is of interest  for various reasons: harmonic domains  are critical points for the torsional rigidity functional; in spectral geometry, they are extremal for a certain Steklov eigenvalue problem for differential forms (see \cite{RS}) and also for a fourth order Steklov problem on functions (see \cite{RS2}). Finally, minimal free boundary immersions (or, more generally, capillary hypersurfaces) in Euclidean balls are in fact harmonic domains (see below).

To our knowledge, no such classification exists, at the moment, even for harmonic domains in the standard sphere. For partial results, we recall that Serrin's rigidity result was extended to domains in the hyperbolic space and in the {\it hemisphere} by Molzon in \cite{Mol} : any harmonic domain there is a geodesic ball.  The method is the same as Serrin's: Alexandrov reflection.  This method breaks down in the whole sphere and the classification problem  is still open there. Note that there are plenty of harmonic domains in $\sphere n$ which are not balls: just take any isoparametric tube in $\sphere n$ (see \cite{Shk} or Section \ref{firsthalfs} below).

We now discuss the relation between harmonicity and the constant flow property. We first remark the following fact. 

\begin{thm}\label{chfhar}{\rm (\cite{S3}, Theorem 10)}. Any domain with the constant flow property is also harmonic. 
\end{thm}

In particular, any isoparametric tube is a harmonic domain thanks to  Theorem \ref{main} and Theorem \ref{chfhar}. The question is if the converse to Theorem \ref{chfhar} holds, that is :

\medskip

{\bf Q1} {\it Is it true that any harmonic domain has the constant flow property ?}

\medskip

Thanks to our classification result (Theorem \ref{main}) the above is equivalent to asking:

\medskip

{\bf Q2} {\it Is it true that any harmonic domain is also an isoparametric tube ?}

\smallskip

The answer to both questions is negative, and we wish to point out here  an interesting class of counterexamples. 
Let $B^n$ be the unit ball in $\real n$. A {\it minimal free boundary hypersurface}  is a minimal hypersurface $\Omega$ of $B^n$ such that $\bd\Omega\subseteq\bd B^n$ and  $\Omega$ meets $\bd B^n$ orthogonally. We will verify in Appendix \ref{mfbi} that 
any free boundary hypersurface  is a harmonic domain.
As minimality implies analiticity, any free boundary hypersurface is an analytic manifold with boundary. Now, it is easy to see that any smooth tube has at most two boundary components (see Proposition \ref{up}), and from the above we conclude that: 

\nero {\it Any minimal free boundary surface in $B^3$ with more than two boundary components is a harmonic domain, but not an isoparametric tube.} 

\medskip

We remark that Fraser and Schoen proved in \cite{FS} that, given any positive integer $k$, there exists a minimal free boundary embedding of a (genus zero) surface with $k$ boundary components into $B^3\subset\real 3$. These domains are therefore harmonic, but can't be isoparametric tubes, which give the desired  counterexamples. 

\smallskip

Finally we spend few words on  the following question:

\medskip

{\bf Q3} {\it Is it true that any harmonic domain in $\sphere n$ is an isoparametric tube ?}

\medskip

In $\sphere 2$ this fact is true under the additional assumption that the domain is simply connected (see \cite {EM}).
In an earlier version of this paper, we conjectured that this fact is true in any dimension.  But recently we learned about a
paper by Fall, Minlend and Weth (see \cite{FMW}) where examples of harmonic domains in $\sphere n$ which are not isoparametric tubes are constructed:  these examples are perturbations of tubular neighborhoods of totally geodesic hypersurfaces. Of course their construction gives another counterexample to questions Q1 and Q2 above. 

\smallskip

In conclusion, it seems that the Serrin condition is, in the general Riemannian framework, rather flexible; a stronger condition is needed to imply some strict rigidity, and we proved in this paper that the constant flow property is one such.


\subsection{Organization of the paper}

The rest of the paper is organised as follows. 

\smallskip

In Section \ref{firsthalfs} we show that any isoparametric tube has the constant flow property, in particular, we generalize (with a different proof)  the results of \cite{Shk} from the sphere to a general Riemannian setting. On any isoparametric tube one can define the class of {\it radial} functions as those which are constant on the equidistants to  the soul $P$. Then, the proof is obtained by using the tool of averaging a function over the equidistants, which enables us to show  the crucial property of isoparametric tubes :  the class of radial functions is invariant under the action of the Laplace-Beltrami operator. 

\smallskip

In Section \ref{secondhalf} we prove that, if the metric of $\Omega$ is analytic, and if $\Omega$ has the constant flow property, then $\Omega$ is an isoparametric tube.
We use here in an important way the results of \cite{S3}, where we proved that the equidistants which are close to the boundary have constant mean curvature. However, to describe the global property of such domains, also at points far from the boundary, one needs to take care of the cut-locus ${\rm Cut}_{\Omega}$  of the normal exponential map at the boundary. The conclusion is that $\Omega$ is an isoparametric tube over the soul $P={\rm Cut}_{\Omega}$.

\smallskip

Finally, in the Appendix we put the proofs of some technical results, to lighten the flow of the exposition.  

\smallskip

{\bf Acknoweldegments.} {\it I am grateful to Sylvestre Gallot for useful discussions and precise remarks}.



\section{Isoparametric tubes have the constant flow property}\label{firsthalfs}

The scope of this section is to prove the first half of the main theorem. 

\begin{thm} \label{firsthalf} Let $\Omega$ be a (not necessarily analytic)  compact manifold with smooth boundary. Assume that $\Omega$ is an isoparametric tube. Then $\Omega$ has the constant flow property. 
\end{thm}

We point out one consequence. The following overdetermined problem is known as {\it Schiffer problem} (D) :
$$
\twosystem
{\Delta u=\lambda u\quad\text{on}\quad\Omega}
{u=0, \,\, \derive u{\nu}=c\ne 0 \quad\text{on}\quad\bd\Omega}
$$
In \cite{S3} it is proved that a domain with the constant flow property supports a solution to the above problem  for infinitely many eigenvalues (see Theorem 11 in \cite{S3}). Then:

\begin{cor} 
Any isoparametric tube supports a solution to the Schiffer problem (D) for infinitely many eigenvalues $\lambda$.
\end{cor}  
The proof of Theorem \ref{firsthalf} is divided in several steps. 


\subsection{Normal coordinates} 
Let $\Omega$ be a smooth tube of radius $R$ around the closed submanifold $P^k$, so that ${\rm codim}(P^k)=n-k$. In particular, ${\rm codim}(P)=1$ corresponds to a hypersurface of $\Omega$. We start by introducing normal coordinates based on $P$.

 Let $U(P)$ be the unit normal bundle of $P$; then, $U(P)$ is locally isometric with $P\times \sphere{n-k-1}$ and we can write an element $\xi\in U(P)$ as a pair:
$$
\xi=(x,\nu(x))
$$
where $x\in P$ and $\nu(x)$ is a unit vector in $T_xM$ normal to $T_xP$. We will often write simply $\nu(x)\in U(P)$, with the understanding that $x$ is the base point of the normal vector $\nu(x)$. Consider the normal exponential map
$
\Phi:[-R,R]\times U(P)\to \Omega
$
defined by
$$
\Phi(r,\nu(x))=\exp_x(r\nu(x)).
$$
Then, $\Phi$ is smooth, and by restriction it gives rise to a  diffeomeorphism
$$
\Phi_1:(0,R]\times U(P)\to \Omega\setminus P.
$$

\nero If $y=\Phi_1(r,\xi)$ we say that $(r,\xi)$ are the {\it normal coordinates} of $y$.

\smallskip

For $\xi=(x,\nu(x))\in U(P)$ we define 
\begin{equation}\label{minuscsi}
-\xi=(x,-\nu(x))\in U(P).
\end{equation}

The map $\xi\mapsto -\xi$ is  an isometry of $U(P)$ and  one sees  that 
\begin{equation}\label{minus}
\Phi(r,\xi)=\Phi(-r,-\xi)
\end{equation}
for all $(r,\xi)\in [-R,R]\times U(P)$.

\medskip

We introduce the smooth function $\theta: [-R,R]\times U(P)\to\reals$ defined by  the identity
\begin{equation}\label{density}
\Phi^{\star}dv_{\Omega}(r,\xi)=\theta(r,\xi)\cdot dr dv_{U(P)},
\end{equation}
where $dv_{U(P)}$ is the Riemannian measure of $U(P)$.
Restricted to $(0,R]\times U(P)$, the function $\theta$ is positive, and gives the density of the Riemannian measure in normal coordinates. So, for any integrable function $f$ on $\Omega$:
$$
\int_{\Omega}f(x)dv_{\Omega}(x)=\int_{U(P)}\int_0^Rf(\Phi(r,\xi))\theta(r,\xi)dr dv_{U(P)}(\xi).
$$
Let us denote by $\rho:\Omega\to\reals$ the distance function to $P$:
$$
\rho(x)=d(x,P).
$$
Then $\rho$ is continuous, and is smooth on $\Omega\setminus P$.
For any fixed $r\in (0,R]$ the set $\rho^{-1}(r)$ is a smooth hypersurface of $\Omega$, which is also called the {\it equidistant} at distance $r$ to $P$. For $x\in \Omega\setminus P$, we denote by $\Sigma_x$ the unique equidistant containing $x$, that is
$$
\Sigma_x=\rho^{-1}(\rho(x)).
$$
Observe that, if $\rho(x)=r>0$, then $\Sigma_x=\Phi(\{r\}\times U(P))$; moreover, the unit vector field $N\doteq \nabla\rho$ is everywhere orthogonal to $\Sigma_x$.  The following facts are well-known. 

\begin{prop}\label{up} \item (a) If ${\rm codim}(P)\geq 2$, or ${\rm codim}(P)=1$ and $P$ is one-sided, then $U(P)$ is connected and so is each equidistant $\Sigma_x$.

\item (b) If ${\rm codim}(P)=1$ and $P$ is two-sided, then $U(P)$, as well as all the equidistants, has two connected components.

\item (c) In particular, any smooth tube over a connected submanifold $P$ has at most two boundary components. 
\end{prop}

Recall that if ${\rm codim}(P)=1$, then $P$ is said to be two-sided if the normal bundle of $P$ is trivial, and one-sided otherwise. If two-sided, one can define a global unit normal vector field on $P$, and  $U(P)$ is isometric to $\{-1,1\}\times P$. We remark that, if the ambient manifold $\Omega$ is orientable, then $P$ is one-sided if and only if it is non-orientable. If $\Omega$ is simply connected, then any closed, embedded hypersurface is automatically orientable hence also two-sided.

We define the shape operator $S:T(\Sigma_x)\to T(\Sigma_x)$ (with respect to the unit normal $N=-\nabla\rho$)  by
$$
S(X)=-\nabla_NX
$$
and the mean curvature function of $\Sigma_x$ by $H=\frac{1}{n-1}{\rm tr}S$.
We make use of the following fact.

\begin{prop} Let $\Omega$ be a smooth tube around $P^k$ and let $\theta$ be the density function as defined in \eqref{density}. Let $x=\Phi_1(r,\xi)$ so that the point $x\in\Omega\setminus P$ has normal coordinates $(r,\xi)$ and $\rho(x)=r$.

\item(a) One has:
$$
-\dfrac{\theta'(r,\xi)}{\theta(r,\xi)}=\Delta\rho(x)=-(n-1)H(x)
$$
where $H(x)$ is the mean curvature at $x$ of the equidistant $\Sigma_x$ containing $x$.

\item (b)
In particular, $\Omega$ is an isoparametric tube if and only if $\theta=\theta(r)$ depends only on the radial coordinate $r$.
\end{prop}

\begin{proof}The assertion (a) follows from a calculation done in \cite{Gal}. From (a) one sees easily that
if $\Omega$ is an isoparametric tube, then $H$ is constant on $\Sigma_x$ and the function $H$ depends only on the distance to $P$;  in normal coordinates it can be written $H=H(r)$ and by integration one sees that $\theta(r,\xi)$ depends only on $r$ and not on $\xi$.
\end{proof}

\subsection{Averaging a function over equidistants}

Let $\Omega$ be a smooth tube around $P$ and let $f\in C^{\infty}(\Omega)$. We say that $f$ is {\it radial} if it depends only on the distance to $P$, that is, if there exists a smooth function $\psi:[0,R]\to\reals$ such that
$$
f=\psi\circ\rho.
$$
Given a function $f$ on $\Omega$, radial or not, we can construct a radial function ${\cal A}f$ simply by averaging $f$ over the equidistants. That is, if $x\in\Omega\setminus P$ we define
$$
{\cal A}f(x)\doteq\dfrac{1}{\abs{\Sigma_x}}\int_{\Sigma_x}f,
$$
while if $y\in P$ we define
$$
{\cal A}f(y)\doteq\dfrac{1}{\abs{P}}\int_Pf.
$$
\nero The function  ${\cal A}f$ is the {\it radialization} of $f$. Clearly $f$ is radial if and only if ${\cal A}f=f$. 

\begin{prop}\label{average} Let $\Omega$ be an isoparametric  tube around $P$, let  $f\in C^{\infty}(\Omega)$ and let ${\cal A} f$ be its radialization. Then:
 
\parte a ${\cal A}f$  is  smooth and radial on $\Omega$.

\parte b The radialization commutes with the Laplacian:  for all $f\in C^{\infty}(\Omega)$ one has 
$
{\cal A}\Delta f=\Delta {\cal A}f.
$

\end{prop}

\begin{proof} We start by proving (a).
We can write
\begin{equation}\label{af}
{\cal A}f=\hat f\circ\rho,
\end{equation}
where $\hat f:[0,R]\to\reals$ is the function:
\begin{equation}\label{hatf}
\hat f(r)=\dfrac{1}{\abs{\rho^{-1}(r)}}\int_{\rho^{-1}(r)}f,
\end{equation}
hence ${\cal A}f$ is radial. 

\medskip

Next, we give the expression of ${\cal A}f$ in normal coordinates. 
Define the smooth function $F:[0,R]\times U(P)\to \reals$ by
$$
F(r,\xi)=f(\Phi(r,\xi));
$$
note that $F$ extends to a smooth function on $[-R,R]\times U(P)$. If $r>0$ one has:
$$
\int_{\level{\rho}{r}}f=\int_{U(P)}F(r,\xi)\theta(r,\xi)\,d\xi
$$
where we have set $d\xi=dv_{U(P)}(\xi)$ for simplicity. 
As the tube is isoparametric, $\theta$ depends only on $r$ and one has:
$$
\int_{\level{\rho}{r}}f=\theta(r)\int_{U(P)}F(r,\xi)\,d\xi.
$$
On the other hand
$
\abs{\level{\rho}{r}}=\theta(r)\abs{U(P)},
$
hence we get the following expression of $\hat f$ for $r>0$:
\begin{equation}\label{hatf}
\hat f (r)=\dfrac{1}{\abs{U(P)}}\int_{U(P)}F(r,\xi)\,d\xi.
\end{equation}
Note that $\hat f$ is defined for $r\in (0,R]$ and is smooth there; as $\rho$ is smooth on $\Omega\setminus P$ we immediately get from \eqref{af} that

\nero ${\cal A}f$ is smooth on $\Omega\setminus P$.

\medskip

It remains to show that ${\cal A}f$, as defined above, extends to a smooth function everywhere on $\Omega$.

\medskip 

As $F(r,\xi)$ extends smoothly to $[-R,R]\times U(P)$, the function $\hat f$
extends smoothly to the interval $[-R,R]$. 
Now:
$$
\hat f(0)=\dfrac{1}{\abs{U(P)}}\int_{U(P)}F(0,\xi)\,d\xi.
$$
But $F(0,\xi)=f(\pi(\xi))$ where $\pi:U(P)\to P$ is the natural projection; then, $F(0,\xi)$ does not depend on $\xi$ but only on the base point; moreover, it is  constant on the fiber, which is isometric to $\sphere d$, with $d=\dim\Omega-\dim P-1$. This gives:
$$
\int_{U(P)}F(0,\xi)\,d\xi=\int_{U(P)}f(\pi(\xi))\,d\xi=\abs{\sphere d}\int_Pf.
$$
Clearly $\abs{U(P)}=\abs{\sphere d}\abs{P}$ and therefore
$$
\hat f(0)=\dfrac{1}{\abs{P}}\int_{P}f.
$$
Now, for any sequence $\{x_n\}$ of points with $\rho(x_n)=r_n>0$ converging to a given point $x\in P$ one has:
$$
\lim_{n\to\infty}{\cal A}f(x_n)=\lim_{n\to\infty}\hat f(r_n)=\hat f(0)=\dfrac{1}{\abs{P}}\int_{P}f={\cal A}f(x).
$$
Thus, ${\cal A}f$ is continuous at all points of $P$.


We now show that ${\cal A}f$ is $C^{\infty}$-smooth also at the points of $P$. First, we observe  that the function $\hat f:[-R,R]\to\reals$ is smooth and even at $0$ : $\hat f(r)=\hat f(-r)$. For that,
we use the identity $\Phi(-r,\xi)=\Phi(r,-\xi)$ which implies that $F(-r,\xi)=F(r,-\xi)$; we also use the fact that the map which sends $\xi$ to $-\xi$ is an isometry of $U(P)$. Then:
$$
\begin{aligned}
\hat f(-r)&=\dfrac{1}{U(P)}\int_{U(P)}F(-r,\xi)\,d\xi\\
&=\dfrac{1}{U(P)}\int_{U(P)}F(r,-\xi)\,d\xi\\
&=\dfrac{1}{U(P)}\int_{U(P)}F(r,\xi)\,d\xi\\
&=\hat f(r)
\end{aligned}
$$
Now, in Appendix \ref{radial} we will show the following fact:

\nero Let $f=\hat f\circ\rho$ be a radial function on the smooth tube $\Omega$.  Assume that $\hat f:[0,R]\to\reals$ is smooth and has vanishing derivatives of odd orders at $0$. Then $f$ is $C^{\infty}$- smooth everywhere on $\Omega$. 
 
\smallskip
Applying the above remark to our situation proves part (a) of the Proposition. 

\smallskip

{\bf Proof of (b): commutation property.} We use the following formula, valid for any smooth function on a smooth tube $\Omega$ and for all $r\in (0,R]$ (for a proof, see Appendix \ref{appone}). \begin{equation}\label{levelint}
\dfrac{d}{dr}\int_{\rho^{-1}(r)}f=\int_{\rho^{-1}(r)}\Big(\scal{\nabla f}{\nabla\rho}-f\Delta\rho\Big)
\end{equation}
For $r\in (0,R]$, the level set $\rho^{-1}(r)$ is a smooth hypersurface, which is the boundary of the domain $\{\rho<r\}$ having $N\doteq -\nabla\rho$ as inner unit normal. As the domain is an isoparametric tube, the mean curvature $H$  is constant on $\rho^{-1}(r)$, say $H=H(r)$, hence $\Delta\rho$ is a radial function which can be written:
$$
\Delta\rho=-\eta\circ\rho
$$
where $\eta(r)=(n-1)H(r)$. 
For example, when $\rho$ is the distance to a point in $\real n$ we have 
$
\eta(r)=\frac{n-1}{r}.
$
Then, \eqref{levelint} becomes:
$$
\dfrac{d}{dr}\int_{\rho^{-1}(r)}f=-\int_{\rho^{-1}(r)}\derive{f}{N}+\eta (r)\int_{\rho^{-1}(r)}f=-\int_{\rho<r}\Delta f+\eta(r) \int_{\rho^{-1}(r)}f
$$
where we have used Green formula in the last step. 
Setting $\psi(r)=\int_{\rho^{-1}(r)}f$ and $V(r)=\abs{\rho^{-1}(r)}$ we see that \eqref{levelint} gives :
$$
\psi'=-\int_{\rho<r}\Delta f+\eta\psi, \quad 
V'=\eta V.
$$
Now, by definition, $\hat f=\psi/V$ hence
$$
\hat f'=-\dfrac{1}{V}\int_{\rho<r}\Delta f, \quad \hat f''=-\dfrac{1}{V}\int_{\rho^{-1}(r)}\Delta f+\dfrac{\eta}{V}\int_{\rho<r}\Delta f,
$$
which can be rewritten:
$$
\hat f''+\eta\hat f'=-\widehat{\Delta f}
$$
On the other hand, 
$$
\Delta(\hat f\circ\rho)=-\hat f''\circ\rho+(\hat f'\circ\rho) \Delta\rho=-(\hat f''+\eta\hat f')\circ\rho.
$$
We conclude that 
$
\Delta(\hat f\circ\rho)=\widehat{\Delta f}\circ\rho,
$
which means precisely, thanks to definition \eqref{af}:
$$
\Delta {\cal A}f={\cal A}\Delta f,
$$
on the set of regular points, that is, on $\Omega\setminus P$. We need to verify this relation also at the points of $P$. But this follows from a standard continuity argument using the fact that ${\cal A}f$ is smooth everywhere and that the commutation relation holds a.e. (that is, on $\Omega\setminus P$). We omit the straightforward details.
\end{proof}

The following consequence is more or less immediate from Proposition \ref{average}.

\begin{cor}\label{rad} Let $\Omega$ be an isoparametric tube, and assume that the function $f_t(x)$ is a solution of the heat equation on $\Omega$ with Dirichlet boundary conditions (and initial condition $f_0$):
$$
\twosystem
{\Delta f_t+\derive{f_t}{t}=0}
{f_t=0\quad\text{on $\bd\Omega$, for all $t>0$.}}
$$
Then the radialization ${\cal A}f_t$ of $f_t$ is the solution of the heat equation on $\Omega$ with Dirichlet boundary conditions and initial condition ${\cal A}f_0$. 

\smallskip

In particular, if the initial condition $f_0$ of $f$ is a radial function, then $f_t$ is radial for all times $t>0$ and consequently $\derive{f_t}{\nu}$ is constant on $\bd\Omega$, for all fixed $t>0$.
\end{cor}


\subsection{Proof of Theorem \ref{firsthalf}}

Assume that $\Omega$ is an isoparametric tube and  consider the temperature function $u_t$ as in \eqref{temp}.  As $u_0=1$ is a radial function, $u_t$ must be  radial for all $t$ thanks to the previous corollary and then $\derive{u_t}{\nu}$ must be constant on the boundary at all times. Thus, $\Omega$ has the constant flow property. 


\section{Geometric rigidity of constant heat flow}\label{secondhalf}

The scope of this section is to prove the second half of the main theorem, that is:  

\begin{thm} Let $\Omega$ be an analytic manifold with smooth boundary. Assume that $\Omega$ has the constant flow property. Then $\Omega$ is an isoparametric tube over a smooth, closed, connected submanifold $P$ of $\Omega$. 
\end{thm}

In this section we denote by $\rho$ the distance function to the boundary of $\Omega$:
$$
\rho(x)={\rm dist}(x,\bd\Omega),
$$
and we let
$$
R=\max_{x\in\Omega}\rho(x)
$$
denote the inner radius of $\Omega$. We denote by ${\rm Cut}_{\Omega}$ the cut-locus of the normal exponential map of $\bd\Omega$ (recalled below).  
It is well-known that ${\rm Cut}_{\Omega}$ is closed in $\Omega$ and has measure zero. 
We will show that, if $\Omega$ has the constant flow property, then ${\rm Cut}_{\Omega}$ is a compact, connected, smooth submanifold of $\Omega$, and that $\Omega$ is a isoparametric tube over ${\rm Cut}_{\Omega}$.

\smallskip

These are the main steps. 

\medskip

{\bf Step 1.} {\it One has that ${\rm Cut}_{\Omega}=\rho^{-1}(R)$, the set of points at maximum distance to $\bd\Omega$}. 

\medskip

If $v$ denotes the mean exit time function (see \eqref{met})  then ${\rm Cut}_{\Omega}$ coincides with the critical set of $v$ and actually ${\rm Cut}_{\Omega}=v^{-1}(m)$, where $m$ is the maximum value of $v$ on $\Omega$ (see Lemma \ref{critical} below).  By flowing $\Omega$ along the integral curves of $\nabla v$, we conclude  that ${\rm Cut}_{\Omega}$ is a deformation retract of $\Omega$, hence it is connected. 

\nero {\it From now on we set $P\doteq {\rm Cut}_{\Omega}$}. Hence $P$ is a closed, connected subset of $\Omega$.

\medskip

We consider the "focal map" $\Phi:\bd\Omega\to\Omega$, defined by
$$
\Phi(y)=\exp_y(R\nu(y)).
$$ 
From Step 1 we see that $\Phi(\bd\Omega)=P$.  
\medskip

{\bf Step 2.} {\it $d\Phi$ has locally constant rank. }

\medskip

Hence, any point $x\in\bd\Omega$ has an open neighborhood $U$ such that $\Phi(U)$ is a smooth submanifold of $\Omega$. The next step is to show the following global result. 

\medskip

{\bf Step 3.} {\it $P$ is a smooth submanifold of $\Omega$.}

\medskip

The final result follows:

\medskip

{\bf Step 4.} {\it $\Omega$ is a an isoparametric tube around $P$.}


\subsection{Some preliminary results} The following preliminary facts apply to any compact manifold $\Omega$ with smooth boundary $\bd\Omega$. One could always think of $\Omega$ as being a domain with smooth boundary in a complete Riemannian manifold $M$ (see for example \cite{PV}).

For $\delta>0$ and small enough we can define the normal exponential map
$\Phi:[0,R+\delta]\times\bd\Omega\to\Omega$ by:
$$
\Phi(r,x)=\exp_x(r\nu(x)).
$$
Define the cut-radius map $c:\bd\Omega\to (0,R]$ as follows: 

\nero for any $x\in\bd\Omega$ the normal geodesic arc $\gamma_x(t)\doteq \exp_x(t\nu(x))$, where $t\in [0,R+\delta]$,  minimizes distance to $\bd\Omega$ if and only if $t\leq c(x)$.

\medskip

The cut-locus  is the set
$$
{\rm Cut}_{\Omega}=\{\Phi(c(x),x): x\in\bd\Omega\}.
$$
Hence a normal geodesic arc, starting at the boundary, minimizes distance to the boundary till it meets the cut-locus, and looses this property immediately after. Let $dv_g$ be the volume form of $\Omega$, and $dv_{\bd\Omega}$ the induced volume form on the boundary. Define a smooth function $\theta:[0,R+\delta]\times \bd\Omega$ by
\begin{equation}\label{definetheta}
\Phi^{\star}dv_g=\theta(r,x)\cdot dr dv_{\bd\Omega}.
\end{equation}
\nero Observe that a point $y=\Phi(r,x)$   is a focal point along the normal geodesic $\gamma_x$ if and only if $\theta(r,x)=0$. A focal point necessarily belongs to the cut-locus. 

\medskip

The distance function to the boundary, denoted $\rho$, is smooth on the regular set
$$
\Omega_{\rm reg}=\Omega\setminus{\rm Cut}_{\Omega},
$$
in particular, near the boundary.  Observe that
$$
\Omega_{\rm reg}=\{\Phi(r,x)\in\Omega: x\in\bd\Omega, r\in [0, c(x))\},
$$
and $\theta$ is positive on $\Omega_{\rm reg}$. On $\Omega_{\rm reg}$ we consider the smooth vector field $\nu=\nabla\rho$ which, restricted to $\bd\Omega$, is the inner unit normal; in general, if $y\in\Omega_{\rm reg}$ then $\nabla\rho(y)$ is a unit normal vector to the equidistant $\Sigma_y=\rho^{-1}(\rho(y))$ through $y$, hence $\Sigma_y\cap\Omega_{\rm reg}$ is a regular hypersurface. On the regular set  one can split the Laplace operator into its normal and tangential parts; precisely, for any smooth function $u$ on $\Omega_{\rm reg}$ one has:
\begin{equation}\label{splitting}
\Delta u(y)=-\deriven 2{u}{\nu}(y)+\eta(y)\derive u\nu(y)+\Delta^Tu(y),
\end{equation}
where  $\eta=\Delta\rho$ and $\Delta^Tu$ is the Laplace operator of the equidistant  $\Sigma_y$ applied to the restriction of $u$ to $\Sigma_y$. Moreover:
$$
\eta(y)=\Delta\rho(y)=\,\text{$(n-1)$ times the mean curvature of $\Sigma_y$ at $y$}
$$
(here the shape operator is the one associated to the unit normal vector $N=\nabla\rho$). 
Finally, a calculation in \cite{Gal} shows that, for all $x\in\bd\Omega$ and $r<c(x)$ one has
\begin{equation}\label{thetaprime}
\eta(\Phi(r,x))=-\dfrac{\theta'(r,x)}{\theta (r,x)}
\end{equation}
where $\theta'$ is differentiation with respect to the variable $r$.  We will use the following fact from our previous paper \cite{S3}. 

\begin{thm}\label{previous} Assume that $\Omega$ has the constant flow property, and let $\eta=\Delta\rho$ (we don't assume that the metric is analytic).  Then, for all $k\geq 0$:
$$
\dfrac{\bd^k\eta}{\bd\nu^k}=c_k\quad\text{on}\quad\bd\Omega,
$$
where $c_k$ is a constant depending only on $k$. 
\end{thm}

\subsection{On the mean exit time function}

We assume from now on that $\Omega$ is a compact manifold with analytic metric and smooth boundary. By the regularity results in \cite{KN} (which we can apply in our case, see \cite{S3}) the boundary is analytic as well. As remarked in \cite{S3} the function $\eta$ is radial on its domain of definition $\Omega_{\rm reg}$. We will give another proof of this fact in Lemma \ref{met} below. 

\medskip

We will draw the following consequence of Theorem \ref{previous}.  Recall the mean-exit time function $v$, solution of the problem:
$$
\twosystem
{\Delta v=1\quad\text{on}\quad\Omega}
{v=0\quad\text{on}\quad\bd\Omega.}
$$
As $\eta$ has normal derivatives of all orders which are constant on $\bd\Omega$, one proves by the local splitting of the Laplacian near the boundary that also $v$ has normal derivatives of all orders which are constant on the boundary. Analyticity will then imply that $v$ is a radial function. 

\begin{lemme}\label{met} \parte a For all $k\geq 0$ and  $x\in\bd\Omega$ one has 
$
\deriven{k}{v}{\nu}(x)=\tilde c_k
$
where  $\tilde c_k$ is a constant depending only on $k$. 

\item b) The function $v$ is radial, and its restriction  to   $\Omega_{\rm reg}$ can be written $v=\psi\circ \rho$ for a smooth function $\psi:[0,R)\to \reals$.
\end{lemme}

For the proof, we refer to Appendix \ref{apmet}. %
\begin{lemme}\label{critical} \item a) $p$ is a critical point of $v$ if and only if $\rho(p)=R$. In other words, the critical set of $v$ coincides with the set of points at maximum distance to the boundary. 

\item b) The function $\eta$ is radial, that is, on $\Omega_{\rm reg}$ one has $\eta=g\circ\rho$ for a smooth function $g:[0,R)\to\reals$.

\item c) The density function $\theta$ is also radial on $\Omega_{\rm reg}$, that is, $\theta=\theta(r)$, and is positive on $[0,R)$. In particular, any focal point must be at maximum distance $R$ to the boundary. 
\end{lemme}

\begin{proof} a) It clearly suffices to show that, if $p$ is a critical point of $v$, then $\rho(p)=R$.  In fact, once we have shown that, we see that any point where $v$ attains its absolute maximum must be at distance $R$;  since $v$ is constant on $\rho^{-1}(R)$ we conclude that any point of $\rho^{-1}(R)$ is a maximum of $v$, hence it is critical. 

\smallskip

Then let $p$ be a critical point of $v$ which is closest to $\bd\Omega$ and set $\rho(p)=r$: now $p$ is an interior point because on the boundary $\abs{\nabla v}=\psi'(0)>0$.  If $\gamma$ is a geodesic arc which minimizes distance from $p$ to the boundary, and if $x\in\bd\Omega$ is the foot of $\gamma$, then $v$ is increasing when moving from $x$ to $p$.
As $v$ is radial we see that the equidistant
$\rho^{-1}(r)$ consists entirely of critical points of $v$. Assume that $\rho^{-1}(r)$ is a regular hypersurface. Then, by Green's formula :
$$
\int_{\rho>r}\Delta v=\int_{\rho^{-1}(r)}\scal{\nabla v}{\nabla\rho}=0,
$$
because $\nabla v=0$ on $\rho^{-1}(r)$. As $\Delta v=1$ one would get $\abs{\rho>r}=0$ which can hold only when $r=\max \rho=R$. 

It remains to prove the lemma when $\rho^{-1}(r)$ is not known to be regular. By assumption, $p$ is a critical point of $v$ closest to the boundary. If $v(p)=a$, let $\{a_n\}$ be any increasing sequence  converging to $a$; obviously each $a_n$ is a regular value of $v$. On the geodesic arc $\gamma$, the function $v$ increases from $0$ to $a$: then, there is a sequence of points $\{p_n\in\gamma\}$ converging to $p$ and such that $v(p_n)=a_n$ for all $n$. Set $\rho(p_n)=r_n$: as $\rho^{-1}(r_n)$ is (possibly, a component of) the regular hypersurface $v^{-1}(a_n)$, it is regular as well. We apply Green formula to the domain $\{\rho>r_n\}$ and get:
\begin{equation}\label{levelset}
\int_{\rho>r_n}\Delta v=\int_{\rho^{-1}(r_n)}\scal{\nabla v}{\nabla\rho}=\abs{\rho^{-1}(r_n)}\psi'(r_n)
\end{equation}
because, by the previous lemma, $v=\psi\circ\rho$ hence $\nabla v=(\psi'\circ\rho)\nabla\rho$.  Note that $\psi'(r_n)\to 0$ as $n\to\infty$. 
 Now $\abs{\rho^{-1}(r_n)}$ is uniformly bounded above by a finite constant depending only on $\Omega$; in fact, standard comparison theorems on the density function $\theta$ show that the volume of any level set of $\rho$  can be controlled in terms of the volume of $\bd\Omega$ and : a lower bound of the mean curvature of $\bd\Omega$, a lower bound of the Ricci curvature of $\Omega$ and the inner radius $R$. Taking the limit as $n\to \infty$ in \eqref{levelset} we obtain as before $\abs{\rho>r}=0$ which implies, again,  that $r=R$.

\medskip

We prove b). From formula \eqref{splitting} one sees that on $\Omega_{\rm reg}$ one has:
$$
\eta\derive v{\nu}=\deriven 2v{\nu}-1
$$
because $\Delta^Tv=0$ ($v$ is radial). 
As $v=\psi\circ \rho$ we have  $\derive v{\nu}=\psi'\circ\rho$ and $\deriven 2v{\nu}=\psi''\circ\rho$. Since the critical set of $v$ is at maximum distance to the boundary, one has $\psi'>0$ on the interval $[0,R)$. From the above equation we get
$$
\eta=\dfrac{\psi''-1}{\psi'}\circ\rho\doteq g\circ\rho
$$
with $g$ smooth on $[0,R)$, showing that $\eta$ is indeed radial. 

\medskip

c) Integrating \eqref{thetaprime} and knowing that  $\theta(0,x)=1$ we see that, for all $x\in\bd\Omega$: 
$$
\theta(r,x)=e^{-\int_0^r\eta(\Phi(s,x))\,ds}=e^{-\int_0^rg(s)\,ds}
$$
the last equality following from b). Hence $\theta$ depends only on $r$.  Finally, pick a point $y$ at maximum distance $R$ to the boundary, and observe that $y=\gamma_x(R)$  for some $x\in\bd\Omega$. Any point $\gamma_x(t)$ with $t\in [0,R)$ is a regular point, hence $\theta(r)\doteq\theta(t,x)>0$ for all $t<R$. 

\end{proof}


\subsection{Proof of Step 1} 

\begin{prop} Let $R$ be the maximum distance of a point of $\Omega$ to the boundary.
Then:
$
{\rm Cut}_{\Omega}=\rho^{-1}(R).
$
\end{prop}

\begin{proof} We first prove that
$
\rho^{-1}(R)\subseteq {\rm Cut}_{\Omega}.
$

\smallskip
In fact, assume to the contrary that $\rho(p)=R$ and $p\notin {\rm Cut}_{\Omega}$. Then, as the cut locus is closed, there is a whole neighborhood $U$ of $p$ not meeting the cut locus. Let $\gamma$ be the unique geodesic segment minimizing the distance from the boundary to $p$ and extend $\gamma$ a little bit beyond $p$.  This extended geodesic segment is still minimizing distance to the boundary, because it does not meet the cut-locus, and it has length greater than $R$. This implies that there are points at distance greater than $R$, which contradicts the assumption. 

\smallskip

It remains to show that
$
 {\rm Cut}_{\Omega}\subseteq \rho^{-1}(R).
$
\smallskip

It is enough to show that if $p$ is a point of the cut-locus which is closest to the boundary then  $\rho(p)=R$. 
It is known that, if $p$ minimizes distance from the cut-locus to the boundary, then there are only two possibilities:

\medskip

1.  either $p$ is a focal point or 

\smallskip

2. $p$ is the midpoint of a geodesic starting and ending at the boundary, and meeting the boundary orthogonally. 

\smallskip

{\it First case.}  This is an immediate consequence of Lemma \ref{critical}, part c) (any focal point is at maximum distance to the boundary). 

\smallskip

{\it Second case.} Assume $\rho(p)=r$.  We parametrize $\gamma$ by arc-length $t$ on the interval $[-r,r]$ so that we have
$$
 \gamma(0)=p, \, \gamma(\pm r)\in\bd\Omega.
$$
We know that $v$ depends only on the distance to the boundary, so that if $\psi:[-r,r]\to\Omega$ is the function
$
\psi(t)=v(\gamma (t))
$
then $\psi$ is even : $\psi(t)=\psi(-t)$ for all $t\in [-r,r]$. Hence $\psi'(0)=0$ and the vectors $\nabla v(\gamma(t))$ and $\gamma'(t)$ are collinear for any $t\in [-r,0)$, which implies 
$$
\abs{\psi'(t)}=\abs{\scal{\nabla v(\gamma(t))}{\gamma'(t)}}=
\abs{\nabla v(\gamma(t))}.
$$
Then:
$$
\abs{\nabla v(p)}=\lim_{t\to 0}\abs{\nabla v(\gamma(t))}=\lim_{t\to 0}\abs{\psi'(t)}=\abs{\psi'(0)}=0
$$
Hence $p$ must be a critical point of $v$ and $\rho(p)=R$ as asserted. 

\end{proof}

\subsection{Proof of Step 2.}  The proof follows an argument  in \cite{Wan}. Step 2 will be a consequence of Claims 1 and 2 below.  
\medskip

{\bf Claim 1. \it Each $y_0\in\bd\Omega$ has a neighborhood $U_0$ such that
$
{\rm rk}(d\Phi(y))\geq {\rm rk}(d\Phi(y_0))
$
for all $y\in U_0$. \rm 

\medskip

For the proof, fix orthonormal frames $(e_1,\dots,e_{n-1})$ in $T_{y_0}\bd\Omega$ (resp. $(E_1,\dots E_n)$ in $T_{\Phi(y_0)}\Omega$) and extend them by parallel trasport in a nhbd $W_{y_0}$ of $y_0$ (resp. $W'$ of $\Phi(y_0)$). In these bases, the matrix of $d\Phi(y)$  depends continuously on $y\in W_{y_0}$. It is clear that, if $W_{y_0}$ is sufficiently small one has
$
{\rm rk}(d\Phi(y))\geq {\rm rk}(d\Phi(y_0))
$
for all $y\in W_{y_0}$, showing the claim. 

\medskip

We now show the reverse inequality. The previous argument shows that, if the rank of $\Phi$ at $y_0$ is maximum (that is, equal to $n-1$), then it will be maximum (hence constant) in a neighborhood of $y_0$.
Then, we can assume that $\Phi(y_0)$ (hence every $y\in\bd\Omega$) is a focal point. 
For $y\in\bd\Omega$ let $\gamma_y[0,t]$ be the geodesic segment of length $t$ starting at $y$ and going in the inner normal direction. By the Morse index theorem, the set of focal points on each finite geodesic segment is discrete; by compactness of $\bd\Omega$, there exists $\epsilon >0$ (independent of $y$) such that the geodesic segment
$$
\alpha_y\doteq\gamma_y[0,R+\epsilon]
$$
will have only one focal point, namely, $\Phi(y)$. Its Morse index ${\rm Ind}(\alpha_y)$ is precisely the null space of $d\Phi(y)$. Consequently :

$$
{\rm Ind}(\alpha_y)=n-1-{\rm rk} (d\Phi(y))
$$
for all $y\in\bd\Omega$.

\medskip

{\bf Claim 2. \it Each $y_0\in\bd\Omega$ has a neighborhood $V_0$ such that
$
{\rm Ind}(\alpha_y)\geq {\rm Ind}(\alpha_{y_0})
$
for all $y\in V_0$. Consequently, on that neighborhood:
$$
{\rm rk} (d\Phi(y))\leq {\rm rk} (d\Phi(y_0)).
$$\rm

For the proof, we observe that the index form on the geodesic $\alpha_y$ depends continuously on $y$. Recall that the index form is a quadratic form:
$$
Q_y:V(\alpha_y)\times V(\alpha_y)\to \reals,
$$
where $V(\alpha_y)$ is the vector space of piecewise-smooth vector fields which are orthogonal to $\alpha_y$ and tangent to $\bd\Omega$ at $y$; the Morse index of $\alpha_y$ is then the index of $Q_y$, and equals the maximal dimension of a subspace of $V(\alpha_y)$ on which $Q_y$ is negative definite.
Now, if $E$ is a $k$-dimensional subspace of $V(\alpha_{y_0})$ on which $Q_{y_0}$ is negative definite, and if $y$ is a point of $\bd\Omega$ near $y_0$, we can parallel transport the vector fields of $E$  to obtain a $k$-dimensional subspace $\tau(E)$ of $V(\alpha_{y})$; as $Q_y$ depends continuously on $y$, it will still be negative definite on $\tau(E)$ provided that $y$ is close enough to $y$. Hence, the index cannot decrease locally, proving the claim. 

\subsection{Proof of Step 3} 

\nero We say that $p$ is a {\it smooth point} of $P$ if there exists an open nghbd $V$ of $p$ in $\Omega$ such that $V\cap P$ is  a smooth $k$-dimensional submanifold of $\Omega$.
Clearly $P$ is a smooth submanifold if and only if every point of $P$ is smooth. 

\medskip

We wish to show that any point $p_0\in P$ is smooth. Fix one such point, and pick $x_0\in\bd \Omega$ such that $p_0=\Phi(x_0)$ (recall that $\Phi$ is surjective). By Step 2,  the rank of $\Phi$ is locally constant; then,  by the constant rank theorem, we can find a neighborhood $U$ of $x_0$ such that $W\doteq\Phi(U)$ is a $k$-submanifold of $\Omega$, and $p_0\in W$. 
Next we claim

\begin{prop}\label{tub} Let $U_{\epsilon}(W)$ denote the open $\epsilon$-nghbd of $W$. Then, if $\epsilon>0$ is small enough one has :
$$
U_{\epsilon}(W)\cap P=W.
$$
\end{prop}

Clearly, Step 3 follows from the above Proposition, because,  taking $V=U_{\epsilon}(W)$, we see that $p_0$ is a smooth point of $P$. We state two lemmas. 


\begin{lemme} Consider the Taylor expansion of $\theta(r)$ at $r=R$:
\begin{equation}\label{taylor}
\theta(r)=c(R-r)^d+O((R-r)^{d+1}), \quad c\ne 0,
\end{equation}
where $d$ is a non-negative integer (the order of vanishing of $\theta$ at $R$). Write $v=\psi\circ\rho$, where $\psi$ is smooth on $[0,R)$. Then, for all $r\in [0,R)$:
$$
\psi'(r)=\dfrac{\int_r^R\theta(s)\,ds}{\theta(r)}\quad\text{and then}\quad 
\lim_{r\to R}\psi''(r)=-\dfrac{1}{d+1}.
$$
\end{lemme}
\begin{proof} For $r<R$ set:
$$
\Omega_r=\{\rho>r\}, \quad\bd\Omega_r=\{\rho=r\}.
$$
Now $\nabla v=(\psi'\circ\rho)\nabla\rho$ and $\nabla\rho$ is the inner unit normal to $\bd\Omega_r$. As
$
\int_{\Omega_r}\Delta v=\int_{\bd\Omega_r}\frac{\bd v}{\bd N}
$
we see that
$$
\psi'(r)=\dfrac{\abs{\Omega_r}}{\abs{\bd\Omega_r}}.
$$
Now 
$$
\abs{\bd\Omega_r}=\int_{\bd\Omega}\theta(r)dv_{\bd\Omega}=\abs{\bd\Omega}\theta(r)\quad\text{and hence}\quad 
\abs{\Omega_r}=\abs{\bd\Omega}\int_r^R\theta.
$$
Then 
$
\psi'(r)=\dfrac{\int_r^R\theta}{\theta(r)}
$
and the calculation of the limit is straightforward from \eqref{taylor}. 
\end{proof}

Again let $\gamma_x:[0,R]\to\Omega$  be the geodesic such that $\gamma_x(0)=x$ and $\gamma'_x(0)=\nu(x)$. Define a  map $\sigma: \bd\Omega\to UN(\Omega)$ by
\begin{equation}\label{sigma}
\sigma(x)=\gamma'_x(R).
\end{equation}
Then $\sigma(x)$ belongs to $T_{\Phi(x)}\Omega$.

\begin{lemme}\label{two} $\sigma(x)$ is an eigenvector of $\nabla^2v$ associated to the eigenvalue $\mu=-\frac{1}{d+1}$, for all $x\in\bd\Omega$.
\end{lemme}
\begin{proof} Note that $\gamma_x$ is an integral curve of $v$. For $r\in [0,R)$ we write $T(r)=\gamma'_x(r)=\nabla\rho(\gamma_x(r))$ and observe that
$
\nabla v=(\psi'\circ\rho)T.
$
Then:
$$
\nabla_T\nabla v=\nabla_T((\psi'\circ\rho)T)=(T\cdot(\psi'\circ\rho)) T+(\psi'\circ\rho)\nabla_TT
$$
Now $T\cdot(\psi'\circ\rho)=\psi''\circ\rho$ and $\nabla_TT=0$ hence
$$
\nabla_T\nabla v=(\psi''\circ\rho)T.
$$
This shows that $T(r)$ is an eigenvector of $\nabla^2v$ associated to the eigenvalue $\psi''(r)$. This holds for all $r<R$, and by continuity it holds also when $r\to R$. As $T(R)=\sigma(x)$, and $\psi''(r)\to \mu$ by the previous lemma, we see
$
\nabla_{\sigma(x)}\nabla v=\mu\sigma(x)
$
and the assertion follows. 
\end{proof}


{\bf Proof of Proposition \ref{tub}.}  

\begin{proof} We retain the notation given before the proposition and recall that, if $m$ is the (absolute) maximum of $v$ in $\Omega$, then $P=v^{-1}(m)$ (Lemma \ref{critical} and Step 1).   The aim is to show that, if $y\in U_{\epsilon}(W)\setminus W$ then $v(y)<m$:
this implies  $y\notin P$ and the proposition follows. 

\smallskip

For all $q\in W$ we have the splitting $T_q\Omega=T_qW\oplus N_qW$, where $N_qW$ is the normal space at $q$. 
Let $E_q(\mu)\subseteq T_q\Omega$ be the eigenspace of $\nabla^2v$ associated to $\mu$. We want to show that
\begin{equation}\label{emu}
N_q(W)\subseteq E_q(\mu).
\end{equation}
In fact, consider the subset of $\bd\Omega$ given by
$$
F=\Phi^{-1}(q)\cap U.
$$
If the rank of $\Phi$ on $U$ is $k$, then $F$  is an $(n-k-1)$-dimensional  open submanifold of $\bd\Omega$. If $x\in F$, then $\sigma(x)=\gamma'_x(R)$ is normal to $W$, because the geodesic $\gamma_x$ is an integral curve of $v$ and $W\subseteq P$, hence $v$ is constant on $W$. Then $\sigma$ restricts to a map
$$
\sigma:F\to UN_q(W).
$$
Since $T_qW$ is $k$-dimensional, we see that $UN_q(W)$ is $(n-k-1)$-dimensional, hence $F$ and $UN_q(W)$ have the same dimension.  By the uniqueness of geodesics, $\sigma$ is injective. Thus, by  invariance of domain:

\nero $\sigma(F)$ is open in $UN_q(W)$.

\medskip

From the previous lemma we know that $\sigma(F)\subseteq UE_q(\mu)$, hence $UE_q(\mu)$ contains an open subset of 
$UN_q(W)$. Now $E_q(\mu)$ is the cone over $UE_q(\mu)$; taking the respective  cones one sees that the subspace $E_q(\mu)$ contains an open subset of $N_q(W)$, hence it must contain the whole of  $N_q(W)$. In conclusion, we showed that for all $y\in W$ one has \eqref{emu}. 

\smallskip

We can now finish the proof. Given  $y\in U_{\epsilon}(W)\setminus W$, let $q\in W$ be the foot of the geodesic minimizing the distance to $W$. We write $y=\gamma_X(t)$ for some $t\in (0,\epsilon)$, where $\gamma_X$ is the geodesic such that
$$
\gamma_X(0)=q, \quad \gamma_X'(0)=X\in UN_q(W).
$$
By \eqref{emu}, $X\in E_q(\mu)$. Let $f_X(t)=v(\gamma_X(t))$. Then:
$$
f_X(0)=m,\quad f_X'(0)=0,\quad f_X''(0)=\nabla^2v(X,X)=\mu=-\dfrac{1}{d+1}<0
$$
and Taylor formula at $t=0$ writes:
$$
f_X(t)=m-\frac{1}{d+1} t^2+O(t^3)
$$
where $O(t^3)$ depends on $X\in UN_q(W) $.  However it is clear, using a compactness argument, that if $\epsilon>0$ is small enough, then 
$
f_X(t)<m
$
for all $0<t<\epsilon$ and $X\in UN(W)$. With that choice of $\epsilon$, one has $v(y)<m$ for all $y\in U_{\epsilon}(W)\setminus W$. The proposition follows. 

\end{proof}


\subsection{Proof of Step 4.} It is enough to show that, for all $r\in (0,R)$ one has:
\begin{equation}\label{equidistant}
\{x\in\Omega: d(x,P)=r\}=\{x\in\Omega: d(x,\bd\Omega)=R-r\}.
\end{equation}
Then, the family of equidistants to $P$ coincides with the family of equidistants to $\bd\Omega$; as each of these is a smooth hypersurface with constant mean curvature the assertion follows. 

\medskip

The proof of \eqref{equidistant} is clear: as $P$ is the set of points at distance $R$ to the boundary, and since $P$ is a smooth submanifold, we see that any point $x\in\Omega\setminus P$ belongs to a unique geodesic arc $\gamma$ meeting $\bd\Omega$ and $P$ orthogonally, and having total length $R$. The geodesic subarcs $\gamma_1\subseteq \gamma$, joining $P$ with $x$, and $\gamma_2\subseteq\gamma$, joining $\bd\Omega$ with $x$, have respective lengths $r$ and $R-r$, and obviously minimize the respective distance. \eqref{equidistant} follows.



\section{Appendix} \label{appendix}

\subsection{Smoothness of radial functions}\label{radial}

Let  $\Omega$ be a smooth tube around the smooth submanifold $P$, and let $\rho:\Omega\to\reals$ be the distance function to $P$. 
\begin{lemme} Consider a radial function $f$ on $\Omega$, such that
$f=\hat f\circ\rho$
where $\hat f:[0,R]\to\reals$ is smooth and has vanishing derivatives of odd orders at zero.  Then $f$ is smooth everywhere on $\Omega$.
\end{lemme}

We first prove the Lemma when $P$ is a $k$-dimensional plane in $\real n$, then we prove the  general case by using Fermi coordinates in a neighborhood of any point of $P$. 

So, let $P$ be a $k$-dimensional plane in $\real n$, where $k=0,\dots, n-1$. We can fix coordinates so that 
$$
P: \, x_{k+1}=\dots =x_n=0,
$$
and therefore 
$$
\rho(x)=\sqrt{x_{k+1}^2+\cdots+x_n^2}.
$$
As $\rho$ is continuous, it is clear that $f(x)=\hat f(\rho(x))$ is also continuous. 
We use the easily proven fact that, under the assumptions on $\hat f$,  the function $\hat F:[0,R]\to\reals$
$$
\hat F(r)=\twosystem{\dfrac{\hat f'(r)}{r}\quad\text{if}\quad r>0}
{\hat f''(0)\quad\text{if}\quad r=0}
$$
is smooth on $[0,R]$ and even at zero. One sees that $\derive{f}{x_i}=0$ everywhere for all $i=1,\dots,k$, and
\begin{equation}\label{smart}
\derive {f}{x_j}=\twosystem
{(\hat F\circ\rho)x_j\quad\text{if}\quad x\notin P, \, j=k+1,\dots,n}
{0\quad\text{if}\quad x\in P}
\end{equation}
which shows that $f$ is $C^1$ everywhere. We now prove any  $f$ as in the hypothesis of the lemma is $C^k$-smooth for all $k$ by induction on $k$. The statement is true for $k=1$; then, assume that the statement is true for the integer $k$.  We apply the inductive hypothesis to  $\hat F\circ\rho$ (we can do that because it is even at $0$);  as  $\hat F\circ\rho$ is $C^k$, equation\eqref{smart} shows that $\derive{f}{x_j}$ is also $C^k$ for all $j$, being the product of two $C^k$ functions. Then $f$ is $C^{k+1}$, as asserted, which completes the induction process : $f$ is $C^{\infty}$-smooth. 

\smallskip

For the extension to the Riemannian case, we use Fermi coordinates which we recall here.
Let $p$ be a point of $P$ and $U$ a neighborhood of $p$ in $P$, on which we can introduce normal coordinates $(x_1,\dots,x_k)$. Let $(e_1,\dots,e_k)$ be an orthonormal basis of $T_pP$, and let  $(\nu_1,\dots,\nu_{n-k})$ be an orthonormal basis of $T^{\perp}_pP$, which we can extend by parallel transport in the normal bundle along any radial geodesics starting at $p$. We thus obtain a local orthonormal frame  $(\nu_1,\dots,\nu_{n-k})$ in $T^{\perp}U$. 

\smallskip

Fix $\epsilon>0$ and small, and consider the open tube $W$ of radius $\epsilon$ around $U$. If $x\in W$, we consider the point $q\in U$ closest to $x$, and assume that it has normal coordinates $(x_1,\dots,x_k)$. If $\epsilon$ is small enough, for each such $x\in W$ there exists a unique vector $\xi\in T^{\perp}_qP$ such that $x=\exp_q\xi$. One can write
$$
\xi=x_{k+1}\nu_1+\dots+x_n\nu_{n-k}.
$$
The Fermi coordinates of $x\in W$ are then, by definition, 
$$
(x_1,\dots,x_k,x_{k+1},\dots,x_n).
$$
Now it is clear that on $W$ we have
$$
P: x_{k+1}=\dots=x_n=0, \quad 
\rho(x)=\sqrt{x_{k+1}^2+\dots+x_n^2}.
$$
If  $f=\hat f(\rho(x))$ is a radial function with the above properties one can apply the argument in Euclidean space and conclude.


\subsection{Proof of formula \eqref{levelint}}\label{appone} Let $\epsilon$ be small and positive, and let 
$$
\Omega_{r,\epsilon}=\{x\in\Omega: r\leq \rho(x)\leq r+\epsilon\}.
$$
Denote by $N$ the inner unit normal to $\bd\Omega_{r,\epsilon}$, so that $N=\nabla\rho$ on 
$\rho^{-1}(r)$ and $N=-\nabla\rho$ on $\rho^{-1}(r+\epsilon)$. Then:
\begin{equation}\label{one}
\begin{aligned}
\int_{\rho^{-1}(r+\epsilon)}f-\int_{\rho^{-1}(r)}f&=-\int_{\rho^{-1}(r+\epsilon)}f\derive \rho{N}-\int_{\rho^{-1}(r)}f\derive \rho{N}\\
&=-\int_{\bd\Omega_{r,\epsilon}}f\derive\rho N\\
&=\int_{\Omega_{r,\epsilon}}\scal{\nabla f}{\nabla\rho}-\int_{\Omega_{r,\epsilon}}f\Delta\rho
\end{aligned}
\end{equation}
By the co-area formula:
\begin{equation}\label{two}
\lim_{\epsilon\to 0}\dfrac 1{\epsilon}\int_{\Omega_{r,\epsilon}}\scal{\nabla f}{\nabla\rho}=
\lim_{\epsilon\to 0}\dfrac 1{\epsilon}\int_r^{r+\epsilon}\Big(\int_{\rho^{-1}(s)}\scal{\nabla f}{\nabla\rho}\Big)\,ds=\int_{\rho^{-1}(r)}\scal{\nabla f}{\nabla\rho}
\end{equation}
Similarly one gets:
\begin{equation}\label{three}
\lim_{\epsilon\to 0}\dfrac 1{\epsilon}\int_{\Omega_{r,\epsilon}}f\Delta\rho=\int_{\rho^{-1}(r)}f\Delta\rho
\end{equation}

From \eqref{one},\eqref{two},\eqref{three} one gets
$$
\lim_{\epsilon\to 0}\Big(\int_{\rho^{-1}(r+\epsilon)}f-\int_{\rho^{-1}(r)}f\Big)=\int_{\rho^{-1}(r)}\scal{\nabla f}{\nabla\rho}-\int_{\rho^{-1}(r)}f\Delta\rho
$$
as asserted. The same argument can be applied when $\epsilon<0$. 


\subsection{Proof of Lemma \ref{met}} \label{apmet}  We first recall some definitions and facts used in the proof of Proposition 17 in \cite{S3}.  Fix $\epsilon>0$ and small so that the collar neighborhood of $\bd\Omega$
$$
U=\{x\in\Omega : \rho(x)<\epsilon\}
$$
does not contain points in the cut-locus. Set $\nu=\nabla\rho$. We say that $\phi\in C^{\infty}(U)$ has {\it level $k$} if $k$ is the largest integer (including possibly $k=+\infty$) such that $\phi,\derive{\phi}{\nu},\dots,\deriven{k}{\phi}{\nu}$ restrict to constant functions on $\bd\Omega$. By convention, if $\phi|_{\bd\Omega}$ is not constant we say that $\phi$ has level $-\infty$; clearly, if $\phi$ is radial then it has level $+\infty$. 

\medskip

By arguing with Taylor expansion along the geodesic exiting a given boundary point, and going in the normal direction, one sees  that $\phi\in C^{\infty}(U)$ has level at least $k$ if and only if there exist smooth functions $\psi:[0,\epsilon)\to \reals$ and $f\in C^{\infty}(U)$ such that one has on $U$:
\begin{equation}\label{level}
\phi=\psi\circ\rho+\rho^{k+1}f.
\end{equation}
This has the following consequences:

\nero If $\phi$ has level at least $k$, then $\derive\phi\nu$ has level at least $k-1$ and $\Delta^T\phi$ has level at least $k$.

\medskip

In fact, the first assertion is clear; for the second, knowing that $\phi$ satisfies \eqref{level} one sees that $\Delta^T\phi=\rho^{k+1}\Delta^Tf$, showing the claim. 

\medskip

We now proceed to prove (by induction on $k$) that $v$ has level at least $k$ for all $k$. This  will imply the first part of the Lemma.

First, observe that, as $\Omega$ has the constant flow property, it is also harmonic by Theorem \ref{chfhar}, hence $\derive{v}{\nu}$ is constant on $\bd\Omega$ and $v$ has level at least one.
The assertion is then true for $k=1$.
Now assume that $v$ has level at least $k$: we need to show that then it has level at least $k+1$.
Recall  the identity
$$
\deriven{2}{v}{\nu}=\eta\derive v{\nu}-1+\Delta^Tv.
$$
We know from Theorem \ref{previous} that $\eta$ has level $+\infty$. Then one sees easily from the above formula that $\deriven{2}{v}{\nu}$ has level at least $k-1$. The identity
$$
\deriven{k+1}{v}{\nu}=\dfrac{\bd^{k-1}}{\bd\nu^{k-1}}\deriven{2}{v}{\nu}
$$
shows that $\deriven{k+1}{v}{\nu}$ has level at least zero, that is, is constant on $\bd\Omega$, hence $v$ has level at least $k+1$ and the induction step is complete. 

\medskip

We then prove b). As $\Omega$ is analytic, with analytic boundary, and since $v$ is a solution of an elliptic equation with analytic coefficients, we see that $v$ is analytic up to the boundary. We fix a point $y\in\bd\Omega$ and the normal geodesic $\gamma_y: [0,R]\to \Omega$ with $\gamma_y(0)=y$ and initial velocity given by $\nu(y)$. 
The function
$
\psi_y(t)=v(\gamma_y(t))
$
is then analytic on $[0,R)$ and one has:
$$
\psi_y(r)=\sum_{k=0}^{\infty}\dfrac{1}{k!}\deriven{k}v{\nu}(y)r^k=\sum_{k=0}^{\infty}\dfrac{\tilde c_k}{k!}r^k\doteq \psi(r)
$$
As the right-hand side is independent on $y$, the value of $v$ at any point at distance $r$ to the boundary is constant, equal to $\psi(r)$. Hence $v$ is radial.


\subsection{Free boundary hypersurfaces are harmonic}\label{mfbi} Let $\Omega$ be a minimal free boundary hypersurface of the unit ball $B^{n+1}$. We choose a unit normal vector $N_{\Omega}$ to $\Omega$ in $\real{n+1}$ and let as usual $\nu$ be the unit normal to $\bd\Omega$ in $\Omega$. Denote the position vector by $x$; this is the radial vector field $x=\sum_{j=1}^{n+1}x_j\dfrac{\bd}{\bd x_j}$. Then, since $\Omega$ meets $\bd B^{n+1}$ orthogonally, we see that $\nu=-x$ on $\bd\Omega$. If $r$ denotes the distance to the origin in $\real{n+1}$, then $x=r\bar\nabla r$, where $\bar\nabla$ is the Levi-Civita connection on $\real{n+1}$.

We want to show that, if $r$ denotes the distance to the origin in $\real{n+1}$ then the function:
$$
f=\dfrac{1}{2n}(1-r^2)
$$
when restricted to $\Omega$, is a solution of
$$
\twosystem
{\Delta f=1\quad\text{on}\quad\Omega}
{f=0, \,\,\derive f{\nu}=\frac 1n\quad\text{on $\bd\Omega$}}
$$
which shows that $\Omega$ is harmonic. Now it is clear that $f=0$ on $\bd\Omega$. Since
$$
\nabla f=-\frac 1n r\nabla r=-\frac 1n x^T,
$$
where $\xi^T$ is the orthogonal projection of $x$ onto $\Omega$, we see that, on $\bd\Omega$:
$$
\derive f{\nu}=\scal{\nabla f}{\nu}=-\frac 1n\scal{x^T}{\nu}=\frac 1n
$$
because on the boundary $x=x^T=-\nu$. It remains to show that $\Delta f=1$. Now 
$\Delta f=-\frac 1n \delta x^T$. Let $\{e_i\}$ be a local orthonormal frame which is $\nabla$-geodesic 
at a given point $x_0$. Then, at $x_0$:
$$
\begin{aligned}
\delta x^T&=-\sum_{i=1}^{n}e_i\cdot\scal{x^T}{e_i}\\
&=-\sum_{i=1}^{n}e_i\cdot\scal{x}{e_i}\\
&=-\sum_{i=1}^{n}\scal{\bar\nabla_{e_i}x}{e_i}-\sum_{i=1}^{n}\scal{x}{\bar\nabla_{e_i}e_i}
\end{aligned}
$$
Now $\bar\nabla_{e_i}x=e_i$ for all $i$; moreover, if $L$ is the second fundamental form, we have
$$
\sum_{i=1}^{n}\bar\nabla_{e_i}e_i=\sum_{i=1}^{n}\nabla_{e_i}e_i+\sum_{i=1}^{n}L(e_i,e_i)=0
$$
because, at the given fixed point,  $\nabla_{e_i}e_i=0$ and, by assumption, $\Omega$ is minimal so that ${\rm tr} L=0$. We conclude that
$\delta x^T=-n$ hence $\Delta f=1$ as asserted.


\subsection{Proof of Theorem \ref{min}} In the sphere the result has been proved by Nomizu (\cite{N}) and in the Riemannian case it has been announced (without proof) in \cite{Wan}. A formal proof  was given by Ge and Tang in \cite{GT}. 

\smallskip

Under some conditions, this minimality phenomenon  seems to hold even when there exists a family of constant mean curvature hypersurfaces condensing to a submanifold $P$ in the sense of \cite{MP}: 
then $P$ has to be minimal even when the members of this family are not necessarily parallel, as in Definition \ref{iso}  (see \cite{MP}). 
 
\smallskip

Finally, we sketch a direct argument,  in the language  of this paper. Recall the density function $\theta(r,\nu)$ which gives the Riemannian measure in normal coordinates around $P$: here $r>0$ is the distance to $P$ and $\nu(x)\in U(P)$ ($x$ is the base point). We remark (without proof) that if ${\rm dim} P=k$, then we have an asymptotic expansion, as $r\to 0$:
$$
\theta(r,\nu(x))=r^{n-k-1}\Big(1-k\scal{H(x)}{\nu(x)}r+O(r^2)\Big)
$$
where $H$ is the mean curvature vector of the immersion of $P$ into $\Omega$.
Now, if the tube is isoparametric then $\theta(r,\nu)$ depends only on $r$ and not on the direction $\nu$ : this forces $\scal{H}{\nu}=0$ for all $\nu\in U(P)$, which in turn can hold only when $H=0$ identically. Then $P$ is minimal.



\end{document}